\documentclass{amsart}
\usepackage[utf8]{inputenc}
\usepackage{fullpage}
\allowdisplaybreaks
\usepackage{amsmath,amssymb,amsthm}

\usepackage{verbatim}
\usepackage{enumerate}
\usepackage{hyperref}
\usepackage{dsfont}

\usepackage[margin=1in]{geometry}

\newcommand{\Sp}{\mathbb{S}}

\newcommand{\C}{\mathbb{C}}

\newcommand{\G}{\mathcal{G}}

\newcommand{\boxb}{\square_b}

\newcommand{\hpq}{\mathcal{H}_{p,q}(\Sp^{2n-1})}

\usepackage{graphicx}

\newtheorem{theorem}{Theorem}
\newtheorem{remark}{Remark}
\newtheorem{lemma}[theorem]{Lemma}
\newtheorem{proposition}[theorem]{Proposition}

\numberwithin{equation}{section}
\numberwithin{figure}{section}
\numberwithin{theorem}{section}

\title[Estimates for the Complex Green Operator on Spheres]{Sobolev and Schatten Estimates for the Complex Green Operator on Spheres}

\author{Elena Kim}
\address[Elena Kim]{Pomona College, Department of Mathematics, 610 N College Ave, Claremont, CA 91711}
\email{elena.kim@pomona.edu}

\author{W. Jacob Ogden}
\address[W. Jacob Ogden]{University of Minnesota, School of Mathematics, 206 Church Street SE,
Minneapolis, MN, 55455, USA}
\email{ogden048@umn.edu}

\author{Tommie Reerink}
\address[Tommie Reerink]{Massachusetts Institute of Technology, Green Hall, 350 Memorial Drive, Cambridge, MA 02139, USA}
\email{reerinkt@mit.edu}

\author{Yunus E. Zeytuncu}
\address[Yunus E. Zeytuncu]{University of Michigan--Dearborn, Department of Mathematics and Statistics, 2048 Evergreen Road, Dearborn, MI 48128, USA}
\email{zeytuncu@umich.edu}

\thanks{This work is supported by the NSF (DMS-1659203). The work of the last author is also partially supported by a grant from the Simons Foundation (\#353525).}

\begin{document}
\maketitle

\begin{abstract}

The complex Green operator $\mathcal{G}$ on CR manifolds is the inverse of the Kohn-Laplacian $\square_b$ on the orthogonal complement of its kernel. In this note, we prove Schatten and Sobolev estimates for $\mathcal{G}$ on the unit sphere $\mathbb{S}^{2n-1}\subset \mathbb{C}^n$. We obtain these estimates by using the spectrum of $\boxb$ and the asymptotics of the eigenvalues of the usual Laplace-Beltrami operator.
\end{abstract}

\section{Introduction}

\subsection{Background}The  unit  sphere $\Sp^{2n-1}$ in $\C^{n}$ is  a  CR  manifold of  hypersurface  type  with the  CR  structure  induced  from  the  ambient space.  The  tangential  Cauchy-Riemann  complex  with  the  operators $\overline{\partial}_b$ and $\overline{\partial}_b^*$ is defined on the  spaces  of  square integrable $(0,q)$-forms $L^2_{(0,q)}(\Sp^{2n-1})$.
The Kohn Laplacian, given by $$\square_b=\overline{\partial}_b\overline{\partial}_b^*+\overline{\partial}_b^*\overline{\partial}_b$$ is a self-adjoint, linear, densely defined, closed operator on $L^2_{(0,q)}(\Sp^{2n-1})$. Much like the Laplace-Beltrami operator on a Riemannian manifold, many geometric properties of CR manifolds can be studied by analyzing the properties of this differential operator. 
The inverse of $\boxb$ (defined on the orthogonal complement of the kernel of $\boxb$ in $L^2_{(0,q)}(\Sp^{2n-1})$) is called the complex Green operator, and denoted by $\mathcal{G}$.
We refer the reader to \cite{CS01} and \cite{Boggess91CR} for detailed definitions for these operators.

In this note, we obtain Sobolev and Schatten estimates
by using the eigenvalues of $\square_b$ on the sphere. The spectrum for any form level on the sphere was originally computed in \cite{Folland} by using unitary representations. A more direct computation by using spherical harmonics at the functions level can be seen in \cite{REU18}.

\subsection{Spherical Harmonics} 
We begin with a quick overview of spherical harmonics. A complex polynomial on $\C^{n}$ can be written as $$f(z,\overline{z})=\sum_{\alpha,\beta}c_{\alpha,\beta}z^\alpha\overline{z}^\beta$$ where $z\in\C^n$, each $c_{\alpha,\beta}\in\C$, and $\alpha,\beta\in\mathbb{N}^n$ are multiindices. By multiindices, we mean that $\alpha=(\alpha_1,\dots,\alpha_n),$ $z^\alpha=\prod_{j=1}^n z_j^{\alpha_j},$ and  $|\alpha|=\sum_{j=1}^n \alpha_j$. A polynomial
$f(z,\overline{z})$ is called homogeneous of bidegree $(p,q)$ if $f(\lambda_1 z,\lambda_2 \overline{z})=\lambda_1^p\lambda_2^qf(z,\overline{z})$ for all $z\neq0$ and $\lambda_i>0$. A twice-differentiable function $f$ is harmonic if $\Delta f=0$, where the Laplacian $\Delta$ on $\C^n$ is defined by $$\Delta f=4\sum_{j=1}^n\frac{\partial^2f}{\partial z_j\partial\overline{z}_j}.$$

The space of harmonic homogeneous polynomials of bidegree $p,q$ on $\C^n$ is denoted $\mathcal{H}_{p,q}(\C^n).$
A spherical harmonic is the restriction of a harmonic complex polynomial on $\C^n$ to $\Sp^{2n-1}$. It is well-known that any polynomial on $\C^n$ agrees with a harmonic polynomial on the sphere. 

The space $\mathcal{H}_{p,q}(\Sp^{2n-1})$ is the space of restrictions to $\Sp^{2n-1}$ of functions in $\mathcal{H}_{p,q}(\C^n).$ Since distinct harmonic polynomials on the ball cannot have the same boundary values, $\mathcal{H}_{p,q}(\C^n) \cong \mathcal{H}_{p,q}(\Sp^{2n-1}).$
Decomposing a function on $\Sp^{2n-1}$ into homogeneous spherical harmonics is analogous to writing the Fourier series decomposition of a function on the circle. The collection of spaces $\mathcal{H}_{p,q}(\Sp^{2n-1} )$ gives a decomposition of $L^2(\Sp^{2n-1})$ into mutually orthogonal subspaces.

\begin{theorem} The spaces $\mathcal{H}_{p,q}(\Sp^{2n-1})$ are pairwise orthogonal, and $$L^2(\Sp^{2n-1})=\bigoplus_{p,q=0}^\infty\mathcal{H}_{p,q}(\Sp^{2n-1}).$$
\end{theorem}
\noindent We refer to \cite{Axler13Harmonic} for more on spherical harmonics and the proof of the last theorem.

A direct computation shows that $\hpq$ is an eigenspace for $\square_b$. We refer to \cite{Folland} and \cite{REU18} for the proof of the next theorem.

\begin{theorem}
The space $\hpq$ is an eigenspace for $\boxb$ with associated eigenvalue $2q(p+n-1).$ 
\end{theorem}

In order to describe the spectrum of $\boxb$, it is also necessary to determine the multiplicity of each eigenvalue. In other words, we have to determine the dimension of the eigenspace $\hpq.$ An inclusion-exclusion principle argument gives the following result. See \cite{Axler13Harmonic} and \cite{Klima} for detailed proofs.

\begin{lemma}
For $p,q\geq1,$ $$ \dim (\mathcal{H}_{p,q}(\Sp^{2n-1}))=\frac{(n-1)(n+p+q-1)}{pq}\binom{n+p-2}{p-1}\binom{n+q-2}{q-1}.$$ Furthermore, 
$$ \dim (\mathcal{H}_{0,q} (\Sp^{2n-1} ))=  \binom{n+q-1}{q}.$$
\end{lemma}

\subsection{Complex Green Operator}

Given a complete description of the spectrum of $\boxb$, it is simple to write down an explicit representation of $\boxb$ in terms of its spectrum. Let $\{e_\ell \}$ be an orthonormal basis for $(\ker \square_b)^\perp$ which consists of eigenfunctions of $\boxb$, $\boxb e_\ell = \lambda_\ell e_\ell$ for each $\ell.$ Then 
$\boxb f = \sum_{\ell} \langle f , e_\ell \rangle \lambda_\ell e_\ell $ whenever the right side converges in $L^2(\Sp^{2n-1}).$ 

 The complex Green operator $\G$ is a compact linear operator on $L^2( \Sp^{2n-1})$ (actually on any strictly pseudoconvex smooth CR manifold \cite{CS01}). If $f \in ( \ker \boxb)^\perp$, then $\G \boxb f = \boxb \G f=f$, where the left side of this identity is understood only formally.  Since the span of $\{e_\ell\}$ is assumed to be orthogonal to the kernel of $\boxb$, the eigenvalue $\lambda_\ell$ is nonzero for each $\ell$. Thus the complex Green operator, the linear operator $\G:L^2(\Sp^{2n-1} ) \to L^2 (\Sp^{2n-1} )$ defined by 
$$\G f =0 \text{ if } f \in \ker \square_b,$$ 
$$ \G f = \sum_\ell \frac{\langle f , e_\ell \rangle }{\lambda_\ell} e_\ell \text{ if } f \in (\ker \boxb ) ^\perp= \{ f \in L^2 ( \Sp^{2n-1}) \: : \: \langle f , g \rangle =0 \text{ for all } g \in \ker \boxb \},$$ is well-defined.

\subsection{Main Results}
The first result of this paper is a characterization of when the Schatten $r$-norm of $\G$ is finite. We prove that, on $\Sp^{2n-1}$, $\| \G \|_r < \infty$ if and only if $r>n$. Similar Schatten estimates for the $\overline{\partial}$-Neumann operator and Hankel operators recently appeared in \cite{Sah}. We present a proof of this statement in the second section.

In section $3$, we turn attention to the modified Poisson equation $\boxb u =f$. The complex Green operator is the solution operator for this equation; given $f \in (\ker  \boxb)^\perp$, $u=\G f + g$ is a weak solution to $\boxb u =f$, where $g \in \ker \boxb$, and $u=\G f$ is the canonical solution in the sense that it minimizes the $L^2$ norm over all solutions. It is natural to ask how many weak derivatives $\G f$ has in $L^2(\Sp^{2n-1} )$ when $f$ is assumed to be in $L^2(\Sp^{2n-1}).$ Kohn proved that the complex Green operator on a class of pseudoconvex CR manifolds satisfies the estimate 
$$ \| \G f \|_{s+1} \leq C \| f \| _{s } $$ for some $C$ depending only on the underlying manifold $M$, where $\| \cdot \| _s$ denotes the norm in the Sobolev space $H^s ( M)$ \cite{Kohn65}. We offer an elementary proof of this result for the complex Green operator for functions on $\Sp^{2n-1}$ by utilizing the explicit spectral representation. Using this method, we are also able to compute the best constants $C$ on the right hand side of the inequality.

\section{Schatten $r$-norms of $\G$}
As mentioned before, $\G$ is a compact linear operator on $L^2( \Sp^{2n-1})$.
As above, let $\{e_\ell\}$ be an orthonormal basis for $(\ker \boxb)^\perp$ consisting of eigenfunctions of $\boxb $ with associated eigenvalues $\lambda_\ell$ Then, for $f\in L^2(\Sp^{2n-1}),$ 
$$\mathcal{G}f = \sum_{\ell}\frac{\langle f, e_\ell \rangle }{\lambda_\ell} e_\ell.$$ 
 Note that $\mathcal{G}$ has the same eigenfunctions as $\boxb$ and that the eigenvalues of $\mathcal{G}$ are the reciprocals of those of $\boxb$. Thus, $\hpq$ is an eigenspace for $\mathcal{G}$ with the associated eigenvalue $\lambda_{p,q}=\frac{1}{2q(p+n-1)}$.  In this section we study the Schatten $r$-norms of $\mathcal{G}$. 

Let $T$ be a compact and positive semi-definite operator from a separable Hilbert space $H$ to itself. Then, for any $r \in [1, \infty )$, define the Schatten $r$-norm of $T$ by 
$$ \| T \| _r= \left ( \sum_{k=0}^\infty  \lambda_k (T)^r  \right ) ^\frac{1}{r}$$ 
where $\lambda_1 (T) \geq \lambda_2(T) \geq \dots \geq \lambda_k(T) \geq \dots \geq 0$ are the eigenvalues of $T$. An operator $T$ has finite Schatten $r$-norm for some $r<\infty $ only if $T$ is compact, so the Schatten norm quantifies the compactness of an operator. We refer to the references within \cite{Sah} for more general studies on the Schatten estimates on various operators.

The following theorem characterizes the values of $r$ such that $\| \G \| _r<\infty $ on $\Sp^{2n-1}.$ 

\begin{theorem} On $\Sp^{2n-1}$, $\| \mathcal{G} \| _r < \infty $ if and only if $r>n.$ \end{theorem}
\begin{proof} By definition, 
$$ \| \mathcal{G} \| _r =\left (  \sum_{k=1}^\infty \lambda_k(\mathcal{G})^r \right ) ^{\frac1r}$$ where $\lambda_1(\mathcal{G}) \geq \dots\geq  \lambda_k(\mathcal{G}) \geq \dots.$ Combining eigenvalues which are the same, this can be rewritten as 
$$ \| \mathcal{G} \| _r^r = \sum_{k=1}^\infty m_k \lambda_k(\mathcal{G})^r $$ where $\lambda_1(\mathcal{G}) > \dots> \lambda_k(\mathcal{G}) >\dots>0$, and $m_k$ is the 
multiplicity of $\lambda_k$. The eigenvalues of $\G$ are $\lambda_{p,q}( {\G})= \lambda_{p,q}  = \frac{1}{2q(p+n-1)}$ 
with multiplicity 
\begin{equation*}
\begin{split}
m_{p,q} &= \frac{ (n-1) ( n+p+q-1)}{pq} \binom{p+n-2}{p-1} \binom{q+n-2}{q-1}\\
&= \frac{ ( n+p+q-1)}{(n-1)!(n-2)!}(p+n-2) \cdots (p+1)(q+n-2)  \cdots (q+1)   
\end{split}
\end{equation*}
(the latter formula holds even when $p=0$). 
Indexing the sum with $p$ and $q$, we have 
$$ \| \G \| _r^r = \sum_{q=1}^\infty \sum_{p=0}^\infty \frac{ m_{p,q} }{(2 q ( p+n-1))^r} .$$
Clearly $$m_{p,q} \leq \frac{ (n+p+q-1) }{(n-1)! (n-2)!} (p+n-2)^{n-2} (q+n-2)^{n-2} ,$$ and $$\lambda_{p,q} < \frac{ 1}{2 pq } $$ when $p>0$. 
Therefore 
$$ \| \G \| _r^r \leq \sum_{q=1}^\infty \left ( \frac{ (q+n-1)^{n-1} }{(2q (n-1))^r ( n-1)!}+ \sum_{p=1}^\infty \frac{( n+p+q-1)(p+n-2)^{n-2} (q+n-2)^{n-2} }{(2pq)^r (n-1)!(n-2)! }\right) .$$
By the elementary integral test (all the sequences of terms are positive and decreasing if $r$ is assumed to be greater than $n-1$), the convergence of this sum is equivalent to the convergence of the integral 
$$ \int _1^\infty \int_1^\infty \frac{( n+p+q-1)(p+n-2)^{n-2} (q+n-2)^{n-2} }{(2pq)^r (n-1)!(n-2)! } \: dp \: dq $$$$+ \int_1^\infty  \frac{ (q+n-1)^{n-1} }{(2q (n-1))^r ( n-1)!} \: dq.$$ The second term is a single integral, so it is easy to see that it converges if and only if $r>n$. Thus we may restrict our attention to the double integral. The convergence will be decided by the terms of highest total degree in the numerator. These terms are $p^{n-1} q^{n-2}$ and $p^{n-2} q^{n-1}$. Since $n$ is fixed, and all other terms in the numerator have lower degree in $p$ and $q$, it suffices to determine the convergence of the integral 
$$ \int_1^\infty \int_1^\infty \frac{ p^{n-1} q^{n-2} + p^{n-2} q^{n-1} }{p^r q^r } \: dp \: dq   .$$ 
The integral can be rewritten as 
$$\int_1^\infty \int_1^\infty \frac{ p^{n-1} q^{n-2} + p^{n-2} q^{n-1} }{p^r q^r } \: dp \: dq = 
\int_1^\infty \frac{ 1}{q^{r-n+2} } \int _1^\infty \frac{ p^{n-1} + q p^{n-2} }{p^r}\: dp \: dq .$$ If $r>n$, then the integral with respect to $p$ converges and 
$$ \int _1^\infty \frac{ 1}{p^{r-n+1}} + \frac{ q}{p^{r-n+2}} \: dp = \frac{ 1}{r-n} + \frac{q}{r-n+1}.$$ Now, the integral with respect to $q$ converges if and only if $r>n$. This shows that if $r>n$, $\| \G \| _r < \infty$. 

It remains to show that if $r \leq n$, then $\| \G \| _r = \infty$. We will show this by estimating $\| \G \| _r $ from below. We have 
$$ m_{p,q} \geq \frac{ (p+q ) p^{n-2} q^{n-2} }{(n-1)! (n-2)!} ,$$ and 
$$ \lambda_{p,q} \geq \begin{cases} \frac{1}{4 nq } & p < n \\ \frac{1} { 4 pq } & p\geq n .\end {cases} $$
Therefore 
$$ \| \G \| _r^r \geq \sum_{q=1}^\infty \sum_{p=n}^\infty  \frac{ (p+q ) p^{n-2} q^{n-2} }{(4pq)^r (n-1)! (n-2)!}.$$
The convergence of this sum is equivalent to the convergence of the integral 
$$\int_1^\infty \int_n^\infty \frac{ p^{n-1} q^{n-2} + p^{n-2} q^{n-1} }{p^r q^r } \: dp \: dq ,$$ which is the same integral as before except for the limits, so this shows that $\| \G \|_r = \infty $ if $r\leq n$. 
\end{proof}

The argument above gives a rough estimate of the size of the $\|\G \| _r $. A reasonable approximation is given by the integration estimates above. Indeed, 
$$ \| \G \|_r^r \simeq 
\frac{ 1}{(n-1)!(n-2)!}\int_1^\infty \bigg ( \int_0^n \frac{(p+q)p^{n-2}q^{n-2}}{(4nq)^r} \:dp $$$$+\int_n^\infty \frac{(p+q)p^{n-2}q^{n-2}}{(4pq)^r} \:dp \bigg )\: dq+ \frac{n}{(2n-2)^r}$$
$$=\frac{r}{4^r (r-n)(r-n+1)n^{r-n}(n-1)(n-1)!(n-2)!}+ \frac{n}{(2n-2)^r}.$$
It can be checked that this approximation at least captures the behavior of $\|G\|_r$ as $r \to n^+$ and as $r \to \infty$.

\section{Sobolev Estimates for $\G$ on Spheres.}
In this section we consider Sobolev estimates for the complex Green operator on the sphere $\Sp^{2n-1}$. The main question at hand is, given $f \in L^2 ( \Sp^{2n-1})$, how many weak derivatives does $\G f $ have in $L^2(\Sp^{2n-1})$? This is a natural question when considering $\G$ as the solution operator for the partial differential equation $\boxb u =f$. 

Sobolev estimates for the complex Green operator were first established by Kohn, who proved that when $M$ satisfies certain pseudoconvexity conditions, the complex Green operator acting on the space of square-integrable $(p,q)$-forms on $M$ gains one weak derivative, see \cite{Kohn65, CS01}.

The main result of this section is to offer a new proof of this estimate in the case of functions on the sphere and to extract some additional information by taking advantage of an explicit representation of the Sobolev norms in this setting. 

Let $\Delta_{\Sp^{2n-1}}$ be the usual Laplace-Beltrami operator on $\Sp^{2n-1}$. To avoid confusion, we consider $\Delta_{\Sp^{2n-1}} $ as a positive operator. The Laplace-Beltrami operator is a self-adjoint operator defined on a dense subspace of $L^2( \Sp^{2n-1})$. Just as with $\boxb$, we can easily write down a formula for $\Delta_{\Sp^{2n-1}} $ given a description of its eigenvalues and eigenfunctions. The eigenspaces of $\Delta_{\Sp^{2n-1}}$ are the spaces of homogeneous spherical harmonics.  

\begin{theorem}The space $\mathcal{H}_k ( \Sp^{2n-1} ) = \bigoplus_{p+q=k } \hpq $ is an eigenspace for $\Delta_{\Sp^{2n-1}}$ with associated eigenvalue $k(k + 2n-2)$. 
\end{theorem}
We refer to \cite{Stein} for the proof of this statement and details on the Laplace-Beltrami operator on spheres. In particular, this theorem implies that every eigenfunction of $\boxb$ is also an eigenfunction of $\Delta_{\Sp^{2n-1} }$. 
Given this description of the spectrum of $\Delta_{\Sp^{2n-1}}$, one can define the operator $(I + \Delta_{\Sp^{2n-1}} )^t $ for any real $t$. Let $\{ e _ \ell\}$ be an orthonormal basis for $L^2( \Sp^{2n-1} ) $ consisting of eigenfunctions of $\Delta_{\Sp^{2n-1} } $ with $\Delta_{\Sp^{2n-1} }e_\ell = \mu _ \ell e _ \ell.$ Then 
$$ (I + \Delta_{\Sp^{2n-1}} )^t f = \sum_{\ell } \langle f , e_\ell \rangle ( 1 + \mu_ \ell )^t e_ \ell$$
whenever the right side converges in $L^2( \Sp^{2n-1} ) $.

The Sobolev space $H^s( \Sp^{2n-1} )$ consisting of functions in $L^2 ( \Sp^{2n-1} ) $ with weak derivatives of order $s$ in $L^2( \Sp^{2n-1} ) $ can be characterized as the space of functions $f$ for which $(I + \Delta_{\Sp^{2n-1}} )^{\frac s 2 } f \in L^2 ( \Sp^{2n-1})$ \cite{Stein}. The norm on $H^s( \Sp^{2n-1} ) $ is defined by 
$$ \| f \| _{ s} = \| ( I + \Delta_{\Sp^{2n-1}} )^\frac{s}{2} f \| _{L^2 } =\left (  \sum_{\ell} |\langle  f, e_\ell \rangle|^2 ( 1 + \mu_\ell )^s \right )^\frac{1}{2}.$$
This formula makes sense for real $s$.

For the remainder of the paper, we assume that $\{ e_ \ell \}$ is an orthonormal basis for $( \ker \boxb )^\perp$ which consists of eigenfunctions of $\boxb$ with associated eigenvalues $\lambda_\ell$. Thus $e_\ell$ is also an eigenfunction of $\Delta_{\Sp^{2n-1}} $ with eigenvalue $\mu_\ell$. Then, for $f \in L^2 ( \Sp^{2n-1} ) $,
$$ \| \G f \| _s^2 = \sum_\ell \frac{ |\langle f, e_ \ell \rangle |^2 }{ \lambda_\ell^2} ( 1 + \mu_\ell )^s.$$ 
The problem is to determine for which $s$ there exists a constant $C$, not depending on $f$, such that 
$$ \| \G f \|_s \leq C \| f \| _{L^2} $$ for all $f \in L^2 ( \Sp^{2n-1} )$, or more generally, for which $s,t$ there exists $C$ such that 
$$ \| \G f \|_{s+t}  \leq C \| f \| _ t$$
for all $f \in H^t( \Sp^{2n-1} ) .$ 

\begin{lemma} There exists a constant $C$ such that $\|\mathcal{G}f\|_s^2 \leq C \|f\|_{L^2}^2$ if and only if $\left \{\frac{(1+\mu_\ell)^{\frac s 2}}{\lambda_\ell}\right\}$ is bounded.  \end{lemma}

\begin{proof} 

Suppose $\left\{\frac{(1+\mu_\ell)^{\frac s 2}}{\lambda_\ell} \right\}$ is bounded. Then there exists $C>0$ such that $\frac{(1+\mu_\ell)^{\frac s 2}}{\lambda_\ell} < \sqrt C$ for all $\ell$. Therefore $ \| \mathcal{G} f \| _{s}^2 =\sum_\ell | \langle f, e_\ell\rangle | ^2  \frac{( 1+ \mu_\ell )^s}{\lambda_\ell^2}\leq C\sum_\ell | \langle f, e_\ell\rangle | ^2 = C \| f \|_{L^2}^2.  $

Conversely, if $\left\{\frac{(1+\mu_\ell)^{\frac s 2}}{\lambda_\ell} \right\} $ is unbounded, then for any $C > 0$, there exists $\ell$ such that $\frac{(1+\mu_\ell)^{ s }}{\lambda_\ell^2} > C$. Let $f=e_\ell$. Then $\| \mathcal{G} f\| _s^2 = \frac{ (1+\mu_\ell)^s}{\lambda_\ell^2 } > C= C \| f \| _{L^2}^2. $
 \end{proof} 
 
 The same argument shows that for any $t$, there exists $C$ such that $ \| \G f \|_{s+t}  \leq C \| f \| _ t$ if and only if $\left \{\frac{(1+\mu_\ell)^{\frac s 2}}{\lambda_\ell} \right \}$ is bounded. Thus it suffices to determine when this sequence of coefficients is bounded.

\begin{proposition} The sequence $\left \{\frac{(1+\mu_\ell)^{\frac s 2}}{\lambda_\ell} \right \}$ is bounded if and only if $s \leq 1.$ \end{proposition}

\begin{proof} Recall that $\mathcal{H}_k ( \Sp^{2n-1} ) $ is an eigenspace for $\Delta_{\Sp^{2n-1} } $ with eigenvalue $\mu(k)=k ( k+ 2n-2)$, and that $\hpq \subset \mathcal{H}_{p+q} ( \Sp^{2n-1} ) $ is an eigenspace for $\boxb$ with eigenvalue $\lambda(p,q)=2 q ( p+n-1)$. Let $$\lambda_{\min} (k)= \min_{ p+q = k, \  q>0} \{ \lambda(p,q)\}.$$ To determine boundedness of $\frac{(1+\mu_\ell)^{\frac s 2}}{\lambda_\ell} $, it suffices to determine the boundedness of $\frac{(1+ \mu(k))^s }{\lambda_{\min}(k)^2}. $ We check that $\lambda_{\min}(k)=2 (k+ n-2)$. Therefore $$\frac{(1+ \mu(k))^s }{\lambda_{\min}(k)^2} = \frac{(k (k + 2n-2 ) +1 )^s}{4 (k + n -2 )^2},$$ which is bounded if and only if $s \leq 1$. 
\end{proof}

This spectral approach to proving Sobolev estimates has the advantage of revealing the smallest possible constant $C_n$ such that $\| \G f \| _{s+1} \leq C_n \|f \|_s $ for all $f \in H^s(\Sp^{2n-1})$. The minimal value of this constant arises as the supremum of the sequence $\left \{\frac{(1+\mu_\ell)^{\frac 1 2}}{\lambda_\ell} \right \}$.

\begin{theorem}
On $\Sp^{2n-1}$, 
$$ \| \G f \|_{s+1} \leq C_n \| f \| _s, $$ 
where $C_2 = 1$ and 
$$ C_n =\frac{\sqrt{n(n-2)}}{2 (n-1)} $$ if $n \geq 3.$ When $n=2$, the above inequality is an equality if and only if $f \in \mathcal{H}_{0,1}( \Sp^{3} )$, and for $n \geq 3$ equality holds if and only if $f \in \mathcal{H}_{n^2-3n,1} ( \Sp^{2n-1} ).$
\end{theorem}

\begin{proof}
Clearly $$C_n^2 = \sup_\ell \frac{1+\mu_\ell}{\lambda_\ell^2}= \sup_{k \geq 1 } \frac{ 1 + \mu(k) }{\lambda_{\min} (k)^2}= \sup_{k \geq 1 } \frac{k (k + 2n-2 ) +1 }{4 (k + n -2 )^2}. $$

Differentiating with respect to $k,$ we see that $\frac{k (k + 2n-2 ) +1 }{4 (k + n -2 )^2}$ has a critical point at $k=n^2-3n+1.$ We first consider $n=2$. In this case the critical point occurs at $k=-1,$ so it is irrelevant. The sequence $\frac{k (k + 2 ) +1 }{4 k^2}$ is decreasing, so when $n=2$ the supremum is 1 and is achieved at $k=1$.

We then consider $n \geq 3$. The critical point at $k=n^2-3n+1$ is the point at which the supremum is achieved, and the value of the supremum is
$$ \frac{1 +\mu(n^2-3n+1) }{\lambda_{\min}(n^2-3n+1)^2 } = \frac{n(n-2) }{4 (n^2 - 2n -1) }.$$
This establishes the values of $C_n$.

To determine the cases of equality, it suffices to recall for which $p,q$ functions in the eigenspace $\hpq$ of $\boxb$ acquire the coefficient $C_n$ when computing $\| \G f \| _{1}.$ The value of $\lambda(p,q)$ is minimized by setting $q=1.$ The above calculations show that the maximum coefficient $C_n$ is achieved in the case $n=2$ when $p+q=k =1$ and in the case $n\geq 3$ when $p+q=k=n^2-3n+1$. Therefore equality is achieved for functions in $\mathcal{H}_{0,1} ( \Sp^3) $ and for functions in $\mathcal{H}_{n^2-3n,1} ( \Sp^{2n-1} )$ for $n \geq 3.$ For all other pairs of $p,q$, the coefficient arising in the computation of the $H^1 ( \Sp^{2n-1} )$ norm will be smaller than $C_n$, which proves the converse. 
\end{proof}

\begin{remark}
We note that the best constants and cases of equality established by the previous theorem may depend on the specific definition of the Sobolev norms.
\end{remark}{}
\section*{Acknowledgements} 
	%We would like to thank the anonymous referee for constructive feedback.
	This research was partially conducted at the NSF REU Site (DMS-1659203) in Mathematical Analysis and Applications at the University of Michigan-Dearborn. We would like to thank the National Science Foundation, National Security Agency, and University of Michigan-Dearborn for their support. 
	
%\bibliographystyle{alpha}
%\bibliography{paper}

\newcommand{\etalchar}[1]{$^{#1}$}

\end{document}